\documentclass[preprint,3p,10pt]{elsarticle}

\usepackage{amssymb,amsthm}
\usepackage{latexsym,amsmath,subfigure,color}

\journal{}

\newcommand{\eps}{\varepsilon}
\newcommand{\epsb}{\varepsilon_\mathrm{b}}
\newcommand{\mub}{\mu_\mathrm{b}}
\newcommand{\sigmab}{\sigma_\mathrm{b}}
\newcommand{\set}[1]{\left\{#1\right\}}

\newcommand{\mD}{\mathrm{D}}
\newcommand{\mE}{\mathbf{E}}
\newcommand{\mH}{\mathbf{H}}
\newcommand{\p}{\partial}
\newcommand{\ma}{\mathbf{a}}
\newcommand{\mf}{\mathbf{f}}
\newcommand{\mg}{\mathbf{g}}
\newcommand{\mr}{\mathbf{r}}
\newcommand{\vt}{\boldsymbol{\theta}}

\theoremstyle{plain}
\newtheorem{thm}{Theorem}[section]

\theoremstyle{remark}
\newtheorem{rem}{Remark}[section]
\newtheorem{ex}{Example}[section]

\begin{document}

\begin{frontmatter}



\title{Direct sampling method for imaging small anomalies: real-data experiments}

\author{Won-Kwang Park\corref{corPark}}
\ead{parkwk@kookmin.ac.kr}
\address{Department of Information Security, Cryptology, and Mathematics, Kookmin University, Seoul, 02707, Korea.}
\cortext[corPark]{First and corresponding author}
\author{Kwang-Jae Lee}
\ead{reolee0122@etri.re.kr}
\author{Seong-Ho Son\corref{corSon}}
\ead{shs@etri.re.kr}
\address{Radio Environment \& Monitoring Research Group, Electronics and Telecommunications Research Institute, Daejeon, 34129, Korea}
\cortext[corSon]{Co-corresponding author}

\begin{abstract}
In this paper, a direct sampling method (DSM) is designed for a real-time detection of small anomalies from scattering parameters measured by a small number of dipole antennas. Applicability of the DSM is theoretically demonstrated by proving that its indicator function can be represented in terms of an infinite series of Bessel functions of integer order, Hankel function of order zero, and the antenna configurations. Experiments using real-data then demonstrate both the effectiveness and limitations of this method.
\end{abstract}


\end{frontmatter}




\section{Introduction}\label{sec:1}
Several studies have revealed that the direct sampling method (DSM) is a fast, stable, and effective imaging technique in inverse scattering problems. It has been investigated for imaging small two-dimensional anomalies \cite{IJZ1,LZ,KLP,P-DSM1} and retrieving three-dimensional objects \cite{IJZ2} for diffusive \cite{CILZ} and electrical impedance \cite{CIZ} tomography applications and for source detection in stratified ocean waveguides \cite{LXZ}. As we have previously observed, DSM only requires one or at most a few incident fields, rather than additional operations, such as singular value decomposition \cite{AILP,PKLS} or solving ill-posed integral equations \cite{KR} or adjoint problems \cite{LR,NS}.

Recent study \cite{P-DSM2} has shown that the DSM is an effective technique for microwave imaging; however, their simulations differ in crucial ways from real-world scenarios: for example, it is difficult to measure scattering parameters when the transmitting and receiving antennas are in the same location. An extended study that considered a more realistic scenario is required and experiments using real-data should be conducted to show that this method could feasibly be used for real-world applications. Therefore, an appropriate DSM is designed herein, and its mathematical structure is analyzed to demonstrate its real-world applicability.

This study is an extension of a previous study \cite{P-DSM2} on analyzing the DSM's structure to a real-world application. Herein, the DSM is designed to handle a real-world problem, and the mathematical structure of its indicator function is analyzed by expressing it in terms of an infinite series of Bessel functions of integer order. Finally, the imaging performance of designed DSM and related factors are discussed.

This paper is organized as follows. Section \ref{sec:2} briefly discusses the scattering parameters in the presence of a set of anomalies and designing of an indicator function of DSM. Section \ref{sec:3} investigates the structure of this indicator function by expressing it in terms of an infinite series of Bessel functions of the first kind of integer order, Hankel function of order zero, and the antenna configuration. With this, we discuss the factors that influence its imaging performance. Section \ref{sec:4} contains experimental results using real-data for multiple anomalies for supporting theoretical result. Section \ref{sec:5} presents our conclusions and future works.

\section{Scattering parameters and direct sampling method}\label{sec:2}
This section briefly introduces the scattering parameters and discusses the design of a DSM indicator function. Here, we assume that multiple small anomalies $\mD_m$ with smooth boundaries $\p\mD_m$, $m=1,2,\cdots,M$, exist in a homogeneous background medium. For simplicity, the $\mD_m$ are assumed to be small balls of radius $\rho_m$ located at $\mr_m$ and denote the collection of all $\mD_m$ by $\mD$. All the anomalies are contained in a homogeneous domain $\Omega$ and are surrounded by $N$ different dipole antennas $\ma_n$, $n=1,2,\cdots,N$, located on a circle of radius $R$. We denote the set of all antennas by $\mathrm{A}$.

Throughout this study, the materials and anomalies are characterized by their dielectric permittivity and electrical conductivity at a given angular frequency $\omega=2\pi f$. The magnetic permeability is set to be a constant $\mu(\mr)\equiv\mu=4\cdot10^{-7}\pi$ for $\mr\in\Omega$ and the permittivity (conductivity) of the background and $\mD_m$ are denoted by $\eps_\mathrm{b}$ ($\sigma_\mathrm{b}$) and $\eps_m$ ($\sigma_m$), respectively. In addition,  the wavenumber $k$ is given by $\omega^2\mu_\mathrm{b}(\eps_\mathrm{b}+i\sigma_\mathrm{b}/\omega)$.

The scattering parameters ($S-$parameters) $\mathrm{S}(n,n')$ are defined as the ratios of the output voltages (or reflected waves) at the $\ma_n$ to the input voltages (or incident waves) at $\ma_{n'}$. The measured data is the scattered field $\mathrm{S}_{\mathrm{scat}}(n, n')$, which is obtained by subtracting the $S-$parameters with and without anomalies, respectively. Then, based on \cite{HSM2}, $\mathrm{S}_{\mathrm{scat}}(n,n')$ is given by
\begin{equation}\label{ScatteringParameter}
  \mathrm{S}_{\mathrm{scat}}(n,n')=\frac{ik^2}{4\omega\mub}\int_{\Omega}\left(\frac{\eps(\mr)-\epsb}{\epsb}+i\frac{\sigma(\mr)-\sigmab}{\omega\sigmab}\right)\mE_{\mathrm{inc}}(\ma_{n'},\mr)\cdot\mE_{\mathrm{tot}}(\mr,\ma_{n})d\mr,
\end{equation}
where $\mE_{\mathrm{inc}}(\ma_{n'},\mr)$ is the incident electric field in a homogeneous medium due to the point current density at $\ma_{n'}$, which satisfies
\[\nabla\times\mE_{\mathrm{inc}}(\ma_{n'},\mr)=-i\omega\mu\mH(\ma_{n'},\mr))\quad\mbox{and}\quad\nabla\times\mH(\ma_{n'},\mr)=(\sigma_\mathrm{b}+i\omega\eps_\mathrm{b})\mE_{\mbox{\tiny inc}}(\ma_{n'},\mr),\]
and $\mE_{\mathrm{tot}}(\mr,\ma_{n})$ be the total field owing to the presence of the anomalies $\mD$ measured at $\ma_n$, which satisfies
\[\nabla\times\mE_{\mathrm{tot}}(\mr,\ma_{n})=-i\omega\mu\mH(\mr,\ma_{n})\quad\mbox{and}\quad\nabla\times\mH(\mr,\ma_{n})=(\sigma(\mr)+i\omega\eps(\mr))\mE_{\mathrm{tot}}(\mr,\ma_{n}),\]
with transmission conditions at the boundaries $\p\mD$. Here, we assume a time-harmonic dependence $e^{i\omega t}$, and $\mH$ denotes the magnetic field.

It is worth noting that, in general real-world experiments, it is impossible to measure $\mathrm{S}_{\mathrm{scat}}(n,n')$ when $n=n'$ because antenna $\ma_{n'}$ is used for signal transmission and the remaining $N-1$ antennas $\ma_n$, $n=1,2,\cdots,N$ with $n\ne n'$ are used for signal reception. Thus, the indicator function introduced in \cite{IJZ1,LZ,P-DSM1,P-DSM2} cannot be applied for imaging of $\mD$.

Motivated by the above, we design an indicator function of DSM for imaging $\mD$ by finding their positions $\mr_m$ based on the measured data $\mathrm{S}_{\mathrm{scat}}(n,n')$ for fixed $n'$. Since the anomalies are assumed to be small, applying the Born approximation to \eqref{ScatteringParameter} yields
\begin{equation}\label{Approximation}
\mathrm{S}_{\mathrm{scat}}(n,n')=\frac{ik^2}{4\omega\mub}\sum_{m=1}^{M}\rho_m^3\chi(\mr_m)\mE_{\mathrm{inc}}(\ma_{n'},\mr_m)\cdot\mE_{\mathrm{inc}}(\mr_m,\ma_{n})+o(\rho_m^3),\quad \chi(\mr_m)=\frac{\eps_m-\epsb}{\epsb}+i\frac{\sigma_m-\sigmab}{\omega\sigmab}.
\end{equation}
Based on this representation, we introduce the indicator function of DSM as follows. For $\mr\in\Omega$,
\begin{equation}\label{IndicatorFunction}
\mathfrak{F}_{\mathrm{DSM}}(\mr,n')=\frac{|\langle \mathrm{S}_{\mathrm{scat}}(n,n'),\mE_{\mathrm{inc}}(\mr,\ma_{n})\rangle_{L^2(\mathrm{A})}|}{||\mathrm{S}_{\mathrm{scat}}(n,n')||_{L^2(\mathrm{A})}||\mE_{\mathrm{inc}}(\mr,\ma_{n})||_{L^2(\mathrm{A})}},
\end{equation}
where
\[\langle \mf(\ma_n),\mg(\ma_n)\rangle_{L^2(\mathrm{A})}=\sum_{n=1,n\ne n'}^{N}\mf(\ma_n)\overline{\mg(\ma_n)}\quad\mbox{and}\quad||\mf(\ma_n)||_{L^2(\mathrm{A})}=(\langle \mf(\ma_n),\mf(\ma_n)\rangle_{L^2(\mathrm{A})})^{1/2}.\]

Notice that the definition of $\mathfrak{F}_{\mathrm{DSM}}(\mr,n')$ is different from the traditional indicator function of DSM. Nevertheless, the maps of $\mathfrak{F}_{\mathrm{DSM}}(\mr,n')$ shown later contain large peaks at $\mr\approx\mr_m$ and small values at $\mr\ne\mr_m$; therefore, the locations of the $\mD_m$ can be identified via this DSM.

\section{Theoretical result}\label{sec:3}
This section investigates the mathematical structure of $\mathfrak{F}_{\mathrm{DSM}}(\mr,n')$ by proving it can be written in terms of Bessel functions of integer order. It should be noted that, since the omnidirectional antennas are arranged perpendicular to the $z-$axis, the mathematical treatment of time-harmonic electromagnetic waves scattered from thin, infinitely-long cylindrical obstacles, (e.g., \cite{PKLS}) indicating that this problem can be regarded as two-dimensional, and the incident field can be expressed as $
\mE_{\mathrm{inc}}(\mr,\ma_{n})=-iH_0^{(1)}(k|\mr-\ma_{n}|)/4$, where $H_0^{(1)}$ is the Hankel function of order zero. Based on this, the following result can be derived.

\begin{thm}[Mathematical structure of the indicator function]
Assume that the total number of antennas, $N$, is small. Let $|\ma_n|=R$, $\vt_n:=\ma_n/|\ma_n|=(\cos\theta_n,\sin\theta_n)$, and $\mr-\mr_m=|\mr-\mr_m|(\cos\phi_m\sin\phi_m)$. If $k|\mr-\ma_n|\gg0.25$ for $\mr\in\Omega$, then $\mathfrak{F}_{\mathrm{DSM}}(\mr,n')$ can be represented as follows:
\begin{equation}\label{StructureDSM}
\mathfrak{F}_{\mathrm{DSM}}(\mr,n')=\frac{|\Phi(\mr)|}{\displaystyle\max_{\mr\in\Omega}|\Phi(\mr)|},
\end{equation}
where
\[\Phi(\mr)\approx\sum_{m=1}^{M}\rho_m^3\chi(\mr_m)H_0^{(1)}(k|\mr_m-\ma_{n'}|)\left(J_0(k|\mr-\mr_m|)+\frac{1}{N-1}\sum_{n=1,n\ne n'}^{N}\sum_{s\in\overline{\mathbb{Z}}_0}i^s J_s(k|\mr-\mr_m|)e^{is(\theta_n-\phi_m)}\right).\]
Here, $\overline{\mathbb{Z}}_0=\mathbb{Z}\cup\set{-\infty,+\infty}\backslash\set{0}$ and $\mathbb{Z}$ denotes the set of integer number.
\end{thm}
\begin{proof}
From \eqref{Approximation}, we can observe that
\begin{multline}\label{Term1}
\langle \mathrm{S}_{\mathrm{scat}}(n,n'),\mE_{\mathrm{inc}}(\mr,\ma_{n})\rangle_{L^2(\mathrm{A})}=\frac{ik^2}{4\omega\mub}\sum_{n=1,n\ne n'}^{N}\sum_{m=1}^{M}\rho_m^3\chi(\mr_m)\mE_{\mathrm{inc}}(\ma_{n'},\mr_m)\cdot\mE_{\mathrm{inc}}(\mr_m,\ma_{n})\overline{\mE_{\mathrm{inc}}(\mr,\ma_{n})}+o(\rho_m^3)\\
\approx-\frac{k^2}{256\omega\mub}\sum_{m=1}^{M}\rho_m^3\chi(\mr_m)H_0^{(1)}(k|\mr_m-\ma_{n'}|)\left(\sum_{n=1,n\ne n'}^{N}H_0^{(1)}(k|\mr_m-\ma_{n}|)\overline{H_0^{(1)}(k|\mr-\ma_{n}|)}\right).
\end{multline}
Since $\mr_m\in\Omega$ and $k|\mr-\ma_n|\gg0.25$ for all $\mr\in\Omega$, substituting the Hankel function's asymptotic form
\[H_0^{(1)}(k|\mr-\ma_n|)=\frac{1+i}{4\sqrt{k\pi}}\frac{e^{ikR}}{\sqrt{R}}e^{-ik\vt_n\cdot\mr}+o\left(\frac{1}{\sqrt{kR}}\right)\]
yields
\begin{align}
\begin{aligned}\label{Term2}
\sum_{n=1,n\ne n'}^{N}H_0^{(1)}(k|\mr-\ma_{n}|)\overline{H_0^{(1)}(k|\mr-\ma_{n}|)}&=\sum_{n=1,n\ne n'}^{N}\frac{1}{8k\pi R}e^{-ik\vt_n\cdot\mr_m}e^{ik\vt_n\cdot\mr}+o\left(\frac{1}{\sqrt{kR}}\right)\\
&\approx\frac{1}{8k\pi R}\sum_{n=1}^{N}e^{ik\vt_n\cdot(\mr-\mr_m)}-\frac{1}{8k\pi R}e^{ik\vt_{n'}\cdot(\mr-\mr_m)}.
\end{aligned}
\end{align}
Since $\vt_n\cdot(\mr-\mr_m)=|\mr-\mr_m|\cos(\theta_n-\phi_m)$ and the Jacobi-Anger expansion 
\[e^{ix\cos\theta}=J_0(x)+\sum_{s\in\overline{\mathbb{Z}}_0}i^s J_s(x)e^{is\theta}\]
holds uniformly, we can obtain
\begin{align}
\begin{aligned}\label{Term3}
\sum_{n=1}^{N}e^{ik\vt_n\cdot(\mr-\mr_m)}&=\sum_{n=1}^{N}e^{ik|\mr-\mr_m|\cos(\theta_n-\phi_m)}=\sum_{n=1}^{N}\left(J_0(k|\mr-\mr_m|)+\sum_{s\in\overline{\mathbb{Z}}_0}i^s J_s(k|\mr-\mr_m|)e^{is(\theta_n-\phi_m)}\right)\\
&=NJ_0(k|\mr-\mr_m|)+\sum_{n=1}^{N}\sum_{s\in\overline{\mathbb{Z}}_0}i^s J_s(k|\mr-\mr_m|)e^{is(\theta_n-\phi_m)}
\end{aligned}
\end{align}
and, similarly,
\begin{equation}\label{Term4}
e^{ik\vt_{n'}\cdot(\mr-\mr_m)}=J_0(k|\mr-\mr_m|)+\sum_{s\in\overline{\mathbb{Z}}_0}i^s J_s(k|\mr-\mr_m|)e^{is(\theta_n'-\phi_m)}.
\end{equation}
Combining \eqref{Term3} and \eqref{Term4}, \eqref{Term2} can be rewritten as
\begin{multline}\label{Term5}
\sum_{n=1,n\ne n'}^{N}H_0^{(1)}(k|\mr-\ma_{n}|)\overline{H_0^{(1)}(k|\mr-\ma_{n}|)}\\
\approx\frac{N-1}{8k\pi R}\left(J_0(k|\mr-\mr_m|)+\frac{1}{N-1}\sum_{n=1,n\ne n'}^{N}\sum_{s\in\overline{\mathbb{Z}}_0}i^s J_s(k|\mr-\mr_m|)e^{is(\theta_n-\phi_m)}\right).
\end{multline}
Finally, substituting \eqref{Term5} into \eqref{Term1} leads us to
\begin{multline*}
\langle \mathrm{S}_{\mathrm{scat}}(n,n'),\mE_{\mathrm{inc}}(\mr,\ma_{n})\rangle_{L^2(\mathrm{A})}\approx-\frac{(N-1)k}{2048R\omega\mub\pi}\sum_{m=1}^{M}\rho_m^3\chi(\mr_m)H_0^{(1)}(k|\mr_m-\ma_{n'}|)\\\times\left(J_0(k|\mr-\mr_m|)+\frac{1}{N-1}\sum_{n=1,n\ne n'}^{N}\sum_{s\in\overline{\mathbb{Z}}_0}i^s J_s(k|\mr-\mr_m|)e^{is(\theta_n-\phi_m)}\right).
\end{multline*}
With this, applying H{\"o}lder's inequality,
\[\langle \mathrm{S}_{\mathrm{scat}}(n,n'),\mE_{\mathrm{inc}}(\mr,\ma_{n})\rangle_{L^2(\mathrm{A})}\leq||\mathrm{S}_{\mathrm{scat}}(n,n')||_{L^2(\mathrm{A})}||\mE_{\mathrm{inc}}(\mr,\ma_{n})||_{L^2(\mathrm{A})},\]
allows us to obtain \eqref{StructureDSM}, thereby completing the proof.
\end{proof}

\begin{rem}
The indicator function structure \eqref{StructureDSM} implies that the imaging performance is highly dependent on 1) the anomaly's size, permittivity, and conductivity; 2) the antenna configuration (total number and arrangement); and 3) the distance between the transmitter $\ma_{n'}$ and anomaly $\mD_m$. Observations 1 and 2 have also been made by other studies \cite{KLP,P-DSM1,P-DSM2}, but observation 3 is also a significant factor that affects imaging performance in real-world applications, as shown below by the experimental results for Example \ref{Ex1}.
\end{rem}

\section{Experimental results}\label{sec:4}
To evaluate this theoretical result and demonstrate the effectiveness and limitations of the designed DSM, we now present the results of experiments with real-world data. For these experiments, we placed $N=16$ dipole antennas in a circle of radius $0.09$m as follows:
\[\ma_n=0.09\mbox{m}\left(\cos\theta_n,\sin\theta_n\right),\quad\mbox{with}\quad\theta_n=\frac{3\pi}{2}-\frac{2\pi(n-1)}{N}.\]
The homogeneous background medium was filled with water, and its permittivity and conductivity were $\eps_\mathrm{b}=78$ and $\sigma_\mathrm{b}=0.2$S/m, respectively. The imaging region $\Omega$ was set to be the interior of a circle of radius $0.085$m, centered at the origin. The $S-$parameters $\mathrm{S}_{\mathrm{scat}}(n,n')$ were measured using a microwave machine developed by the Electronics and Telecommunications Research Institute (ETRI), and the incident field data $\mE_{\mathrm{inc}}(\ma_{n},\mr)$ used to calculate the indicator function $\mathfrak{F}_{\mathrm{DSM}}(\mr,n')$ via \eqref{IndicatorFunction}, were generated by CST STUDIO SUITE, as shown in Figure \ref{IncidentFieldCST}.

\begin{figure}[h]
\begin{center}
\includegraphics[width=0.37\textwidth]{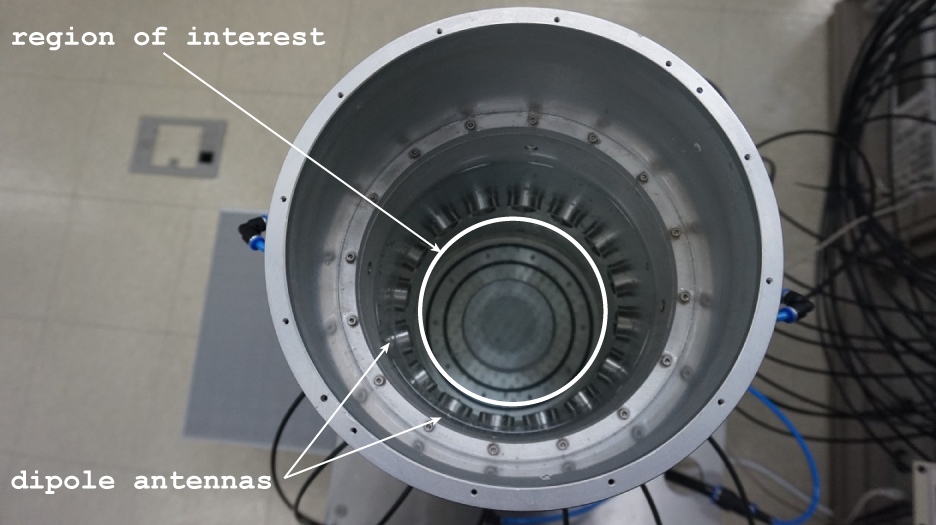}
\includegraphics[width=0.3\textwidth]{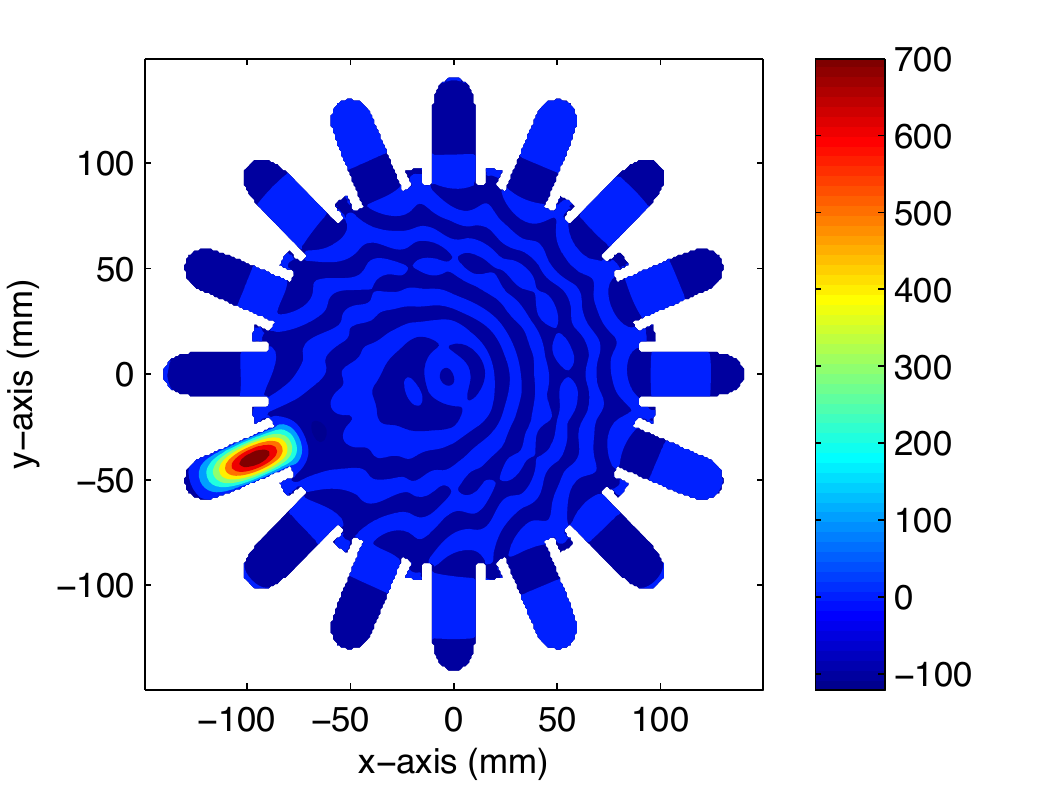}
\includegraphics[width=0.3\textwidth]{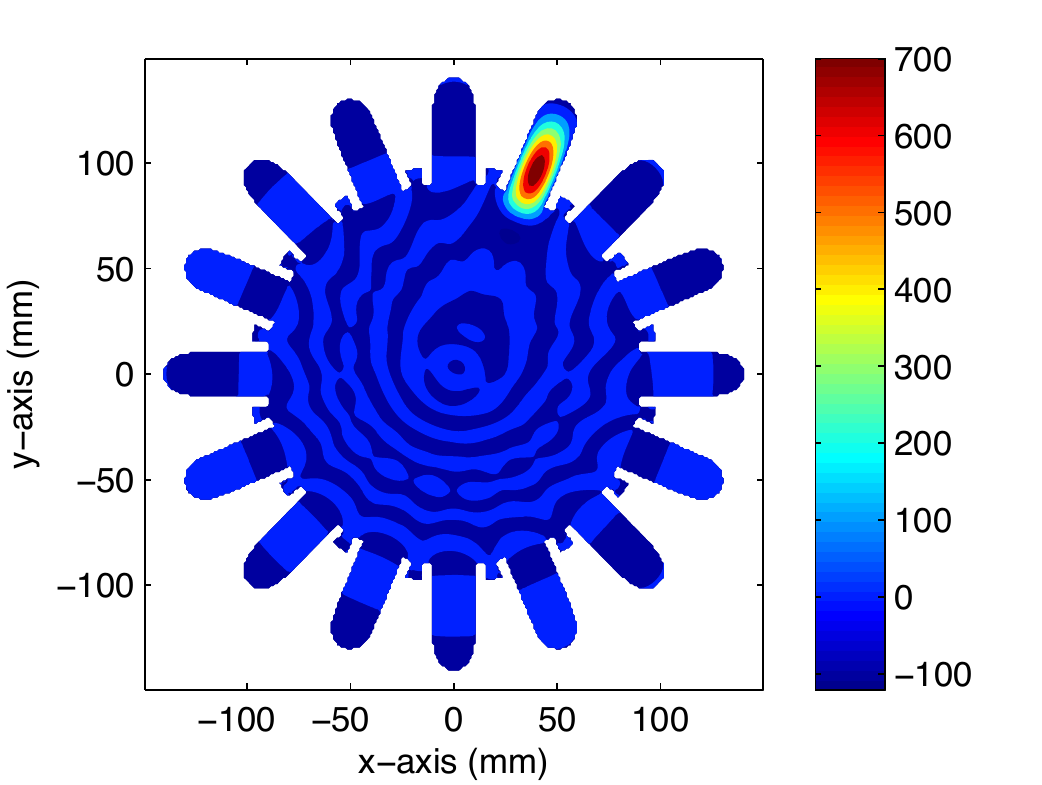}
\caption{\label{IncidentFieldCST}Microwave machine (left) and real part of $\mE_{\mathrm{inc}}(\ma_{n},\mr)$ for $n=4$ (center), and $n=10$ (right).}
\end{center}
\end{figure}

\begin{ex}\label{Ex1}
Here, we consider imaging two anomalies, $\mD_1$ and $\mD_2$, both of diameter $0.0064$m. Maps of $\mathfrak{F}_{\mathrm{DSM}}(\mr,n')$ for several different $n'$ are shown in Figure \ref{Result-Multiple}. Note that the locations of $\mD_1$ and $\mD_2$ can be identified, but some large artifacts also exist. It is also interesting to observe that $\mD_1$ cannot be identified when $n'=9$ and $n'=13$, i.e., when $\mD_1$ is not close enough to $\ma_{n'}$. While $\mD_1$ can be located when $n'=1$, $n'=2$, and $n'=5$, these provide the wrong location for $\mD_2$ because  $\mD_2$ is not close enough to $\ma_{n'}$.
\end{ex}

\begin{figure}[h]
\begin{center}
\includegraphics[width=0.325\textwidth]{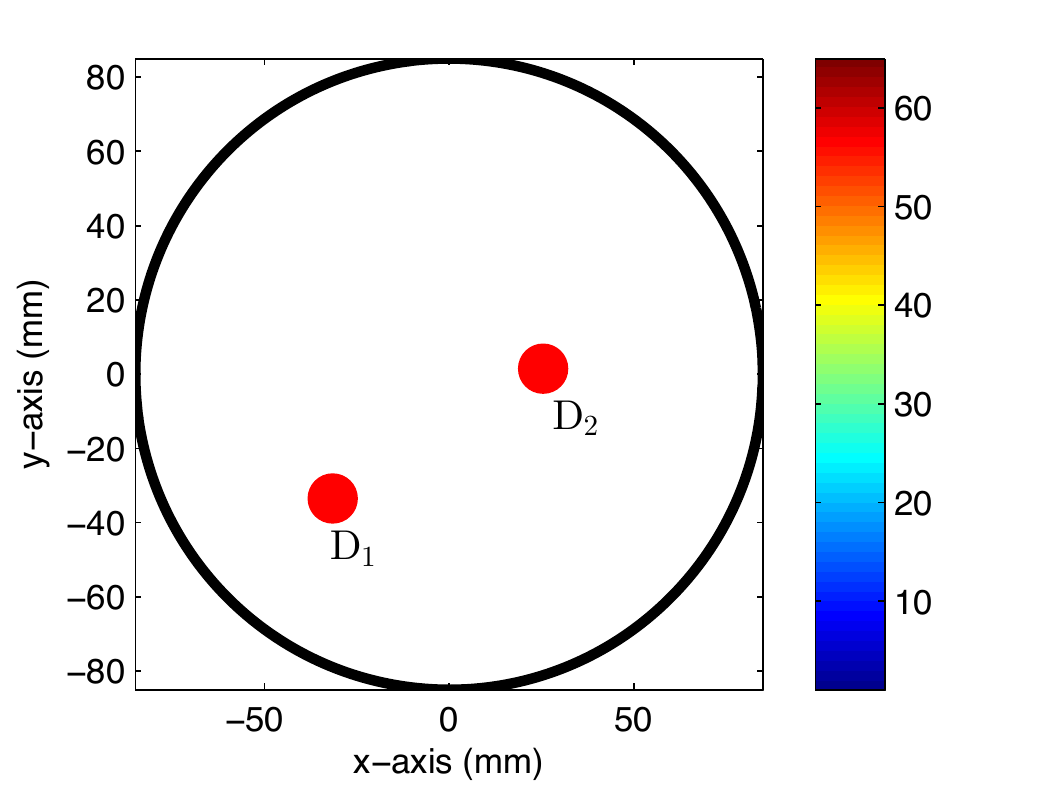}
\includegraphics[width=0.325\textwidth]{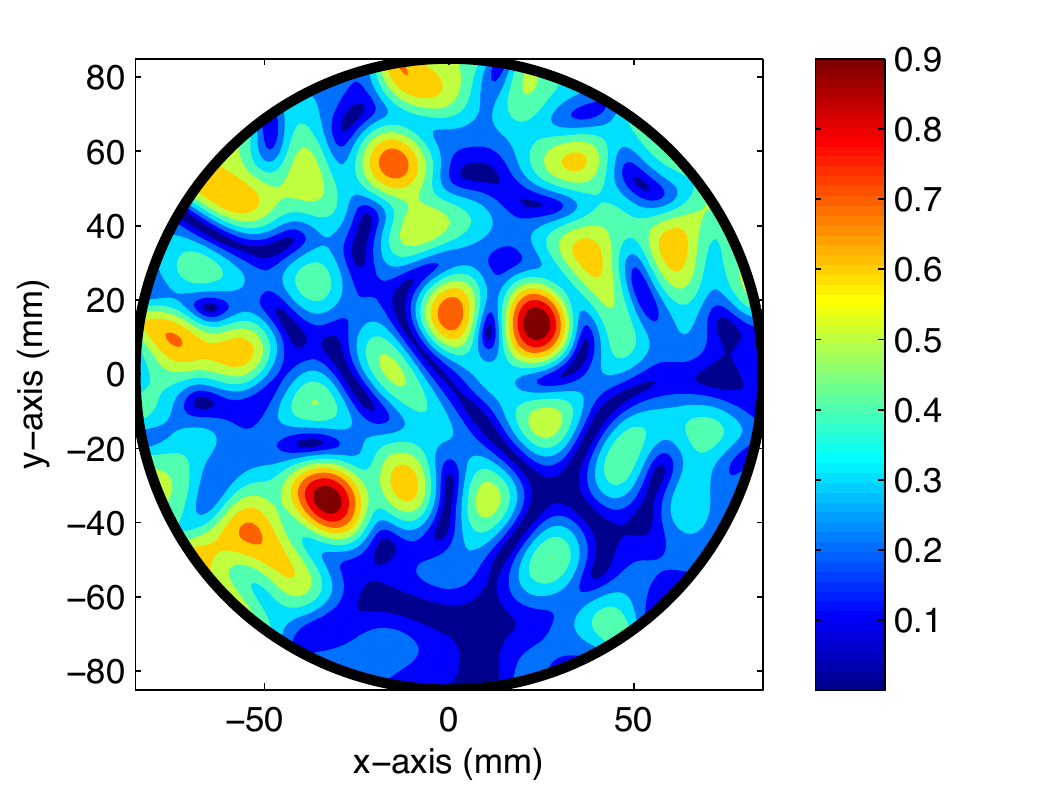}
\includegraphics[width=0.325\textwidth]{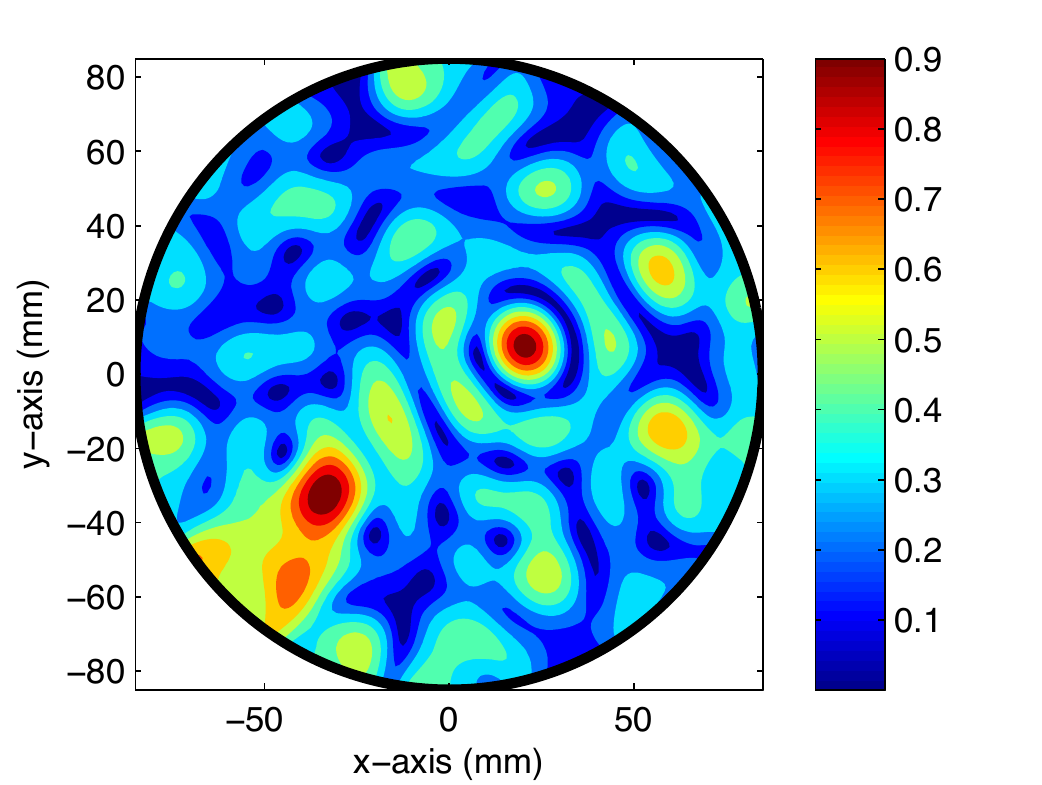}\\
\includegraphics[width=0.325\textwidth]{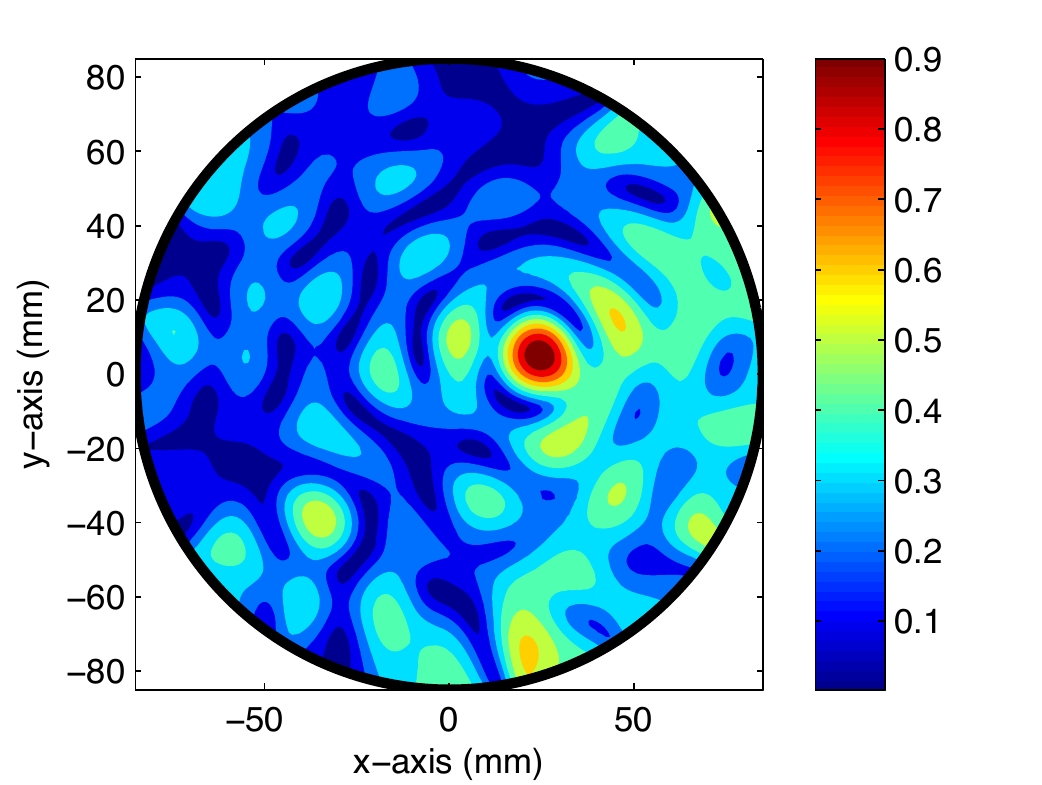}
\includegraphics[width=0.325\textwidth]{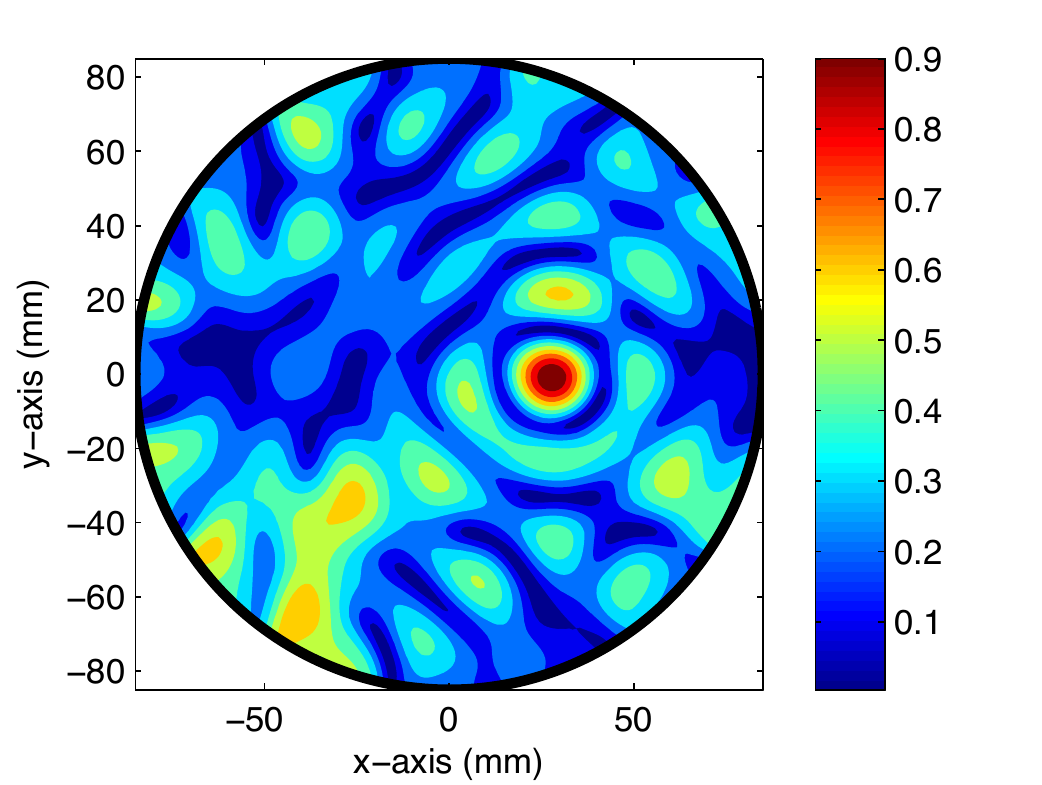}
\includegraphics[width=0.325\textwidth]{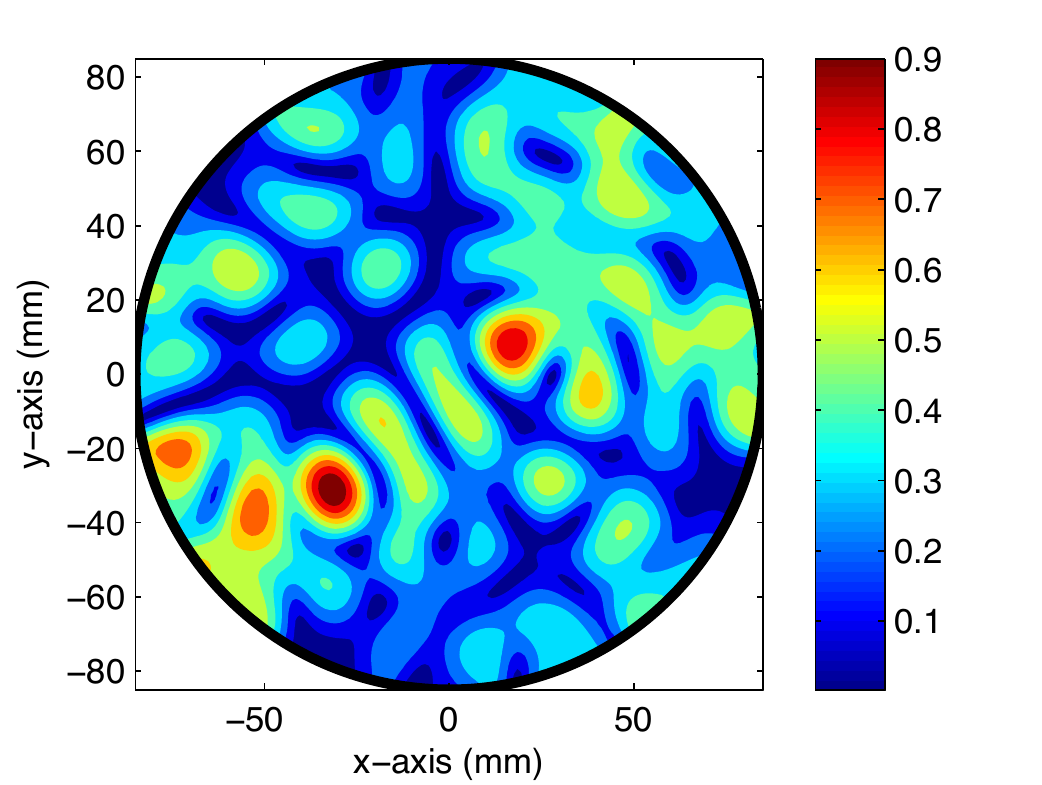}
\caption{\label{Result-Multiple}True locations of the anomalies (top left) and maps of $\mathfrak{F}_{\mathrm{DSM}}(\mr,n')$ for $n'=1$ (top center), $n'=5$ (top right), $n'=9$ (bottom left), $n'=13$ (bottom center), and $n'=2$ (bottom right).}
\end{center}
\end{figure}

\begin{ex}\label{Ex2}
As observed in Example \ref{Ex1}, the imaging performance is highly dependent on the location of the transducer $\ma_{n'}$. To identify all the anomalies in Example \ref{Ex1}, the following indicator function with multiple incident fields is adopted, introduced in \cite{IJZ1,LZ}. For each $\mr\in\Omega$,
  \[\mathfrak{F}_{\mathrm{MDSM}}(\mr):=\max_{n'}\set{\mathfrak{F}_{\mathrm{DSM}}(\mr,n')}.\]
Figure \ref{Result-Multiple-Incident} shows the maps of $\mathfrak{F}_{\mathrm{MDSM}}(\mr)$ for four ($n'=1,5,9,13$), eight ($n'=1,3,5,\cdots,15$), and sixteen ($n'=1,2,3,\cdots,16$) different incident fields. These results indicate that the location of $\mD_1$ cannot be identified with a small number of incident directions but can be identified when sufficient incident directions are used. However, some artifacts still appear in the map; therefore, further improvement is still needed.
\end{ex}

\begin{figure}[h]
\begin{center}
\includegraphics[width=0.325\textwidth]{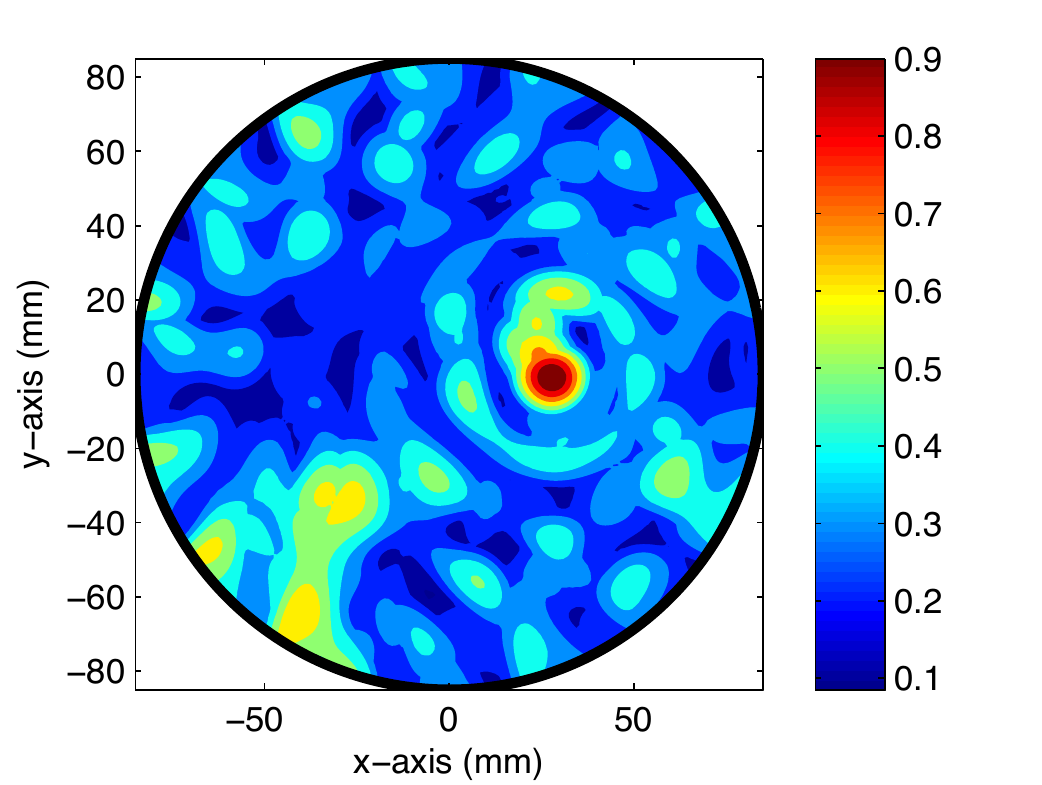}
\includegraphics[width=0.325\textwidth]{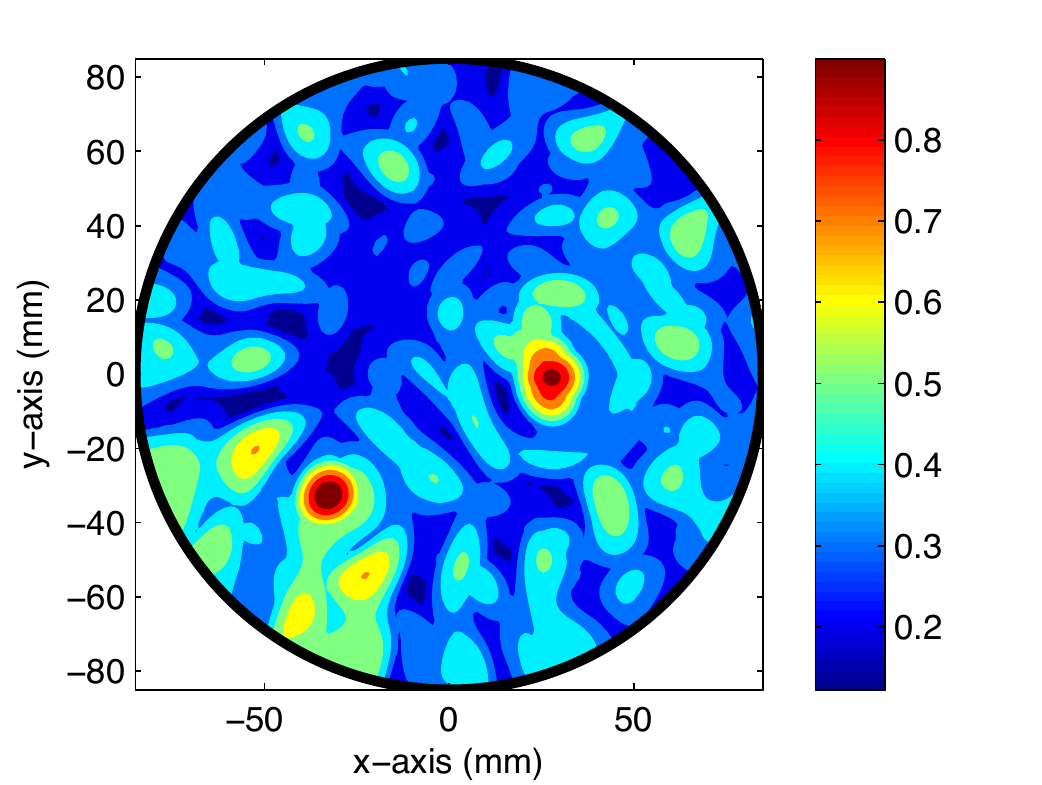}
\includegraphics[width=0.325\textwidth]{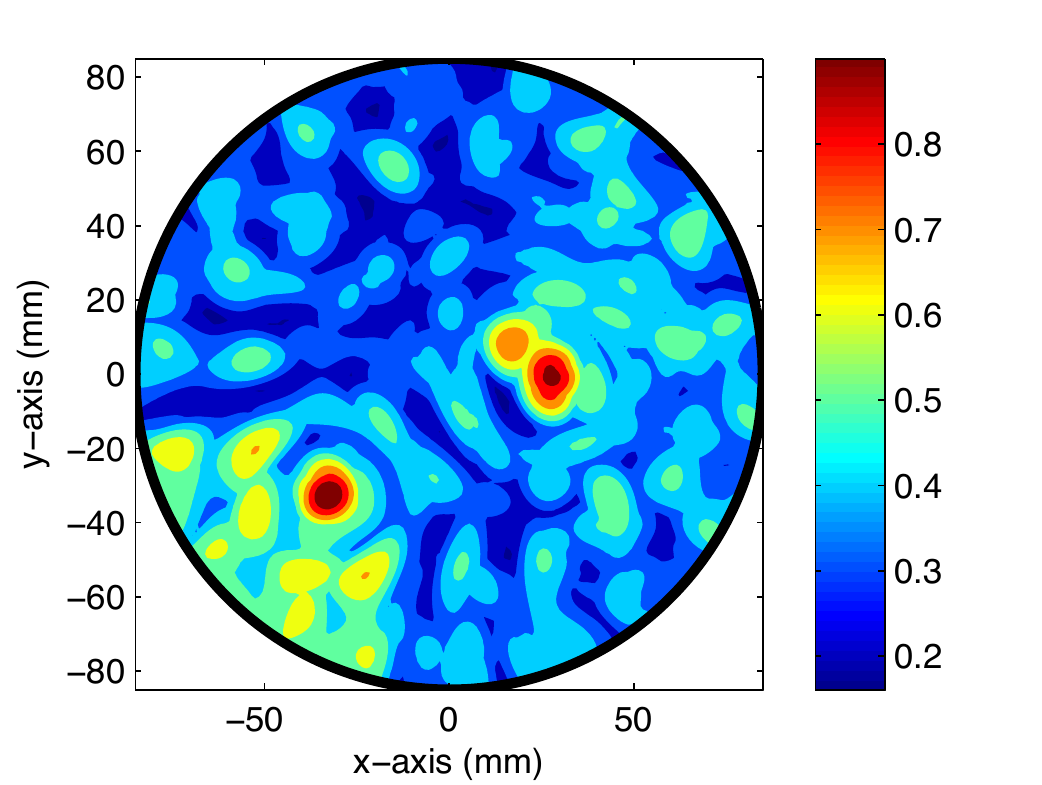}
\caption{\label{Result-Multiple-Incident}Maps of $\mathfrak{F}_{\mathrm{MDSM}}(\mr)$ for four (left), eight (center), and sixteen (right) different incident fields.}
\end{center}
\end{figure}

\section{Conclusions}\label{sec:5}
An indicator function of DSM is designed for identifying small anomalies from scattering parameters measured by a small number of dipole antennas. The experimental results with real-data indicate that such small anomalies can be identified via this DSM with only one or at most a few incident fields, which means that it can be regarded as a fast imaging technique for real-world microwave imaging. Unfortunately, these results also include some artifacts, and the imaging performance depended significantly on the transmitter location. This has motivated us to develop an improved DSM, and this will be the subject of future study.

\section*{Acknowledgements}
This research was supported by the Basic Science Research Program of the National Research Foundation of Korea (NRF) funded by the Ministry of Education (No. NRF-2017R1D1A1A09000547), the research program of Kookmin University in Korea, and Electronics and Telecommunications Research Institute (ETRI) grant funded by the Korean government (No. 18ZR1230).

\bibliographystyle{elsarticle-num-names}
\bibliography{../../../References}
\end{document}